\documentclass[12pt]{article}
\usepackage{amsmath,amssymb, amsfonts, amsthm,amscd}
\usepackage[T2A]{fontenc}
\usepackage[cp1251]{inputenc}
\usepackage[english]{babel}
\usepackage{graphicx}

\theoremstyle{plain}
\newtheorem{theorem}{Theorem}

\newtheorem{proposition}{Proposition}

\theoremstyle{definition}
\newtheorem{definition}{Definition}

\begin{document}

\begin{center}
{\huge Diagonal reduction of matrices over Bezout rings of stable range 1 with the Kazimirsky condition}
\end{center}
\vskip 0.1cm \centerline{{\Large Bohdan Zabavsky, \; Oleh Romaniv}}

\vskip 0.3cm

\centerline{\footnotesize{Department of Mechanics and Mathematics, Ivan Franko National University}}
\centerline{\footnotesize{Lviv, 79000, Ukraine}}
 \centerline{\footnotesize{zabavskii@gmail.com, \; oleh.romaniv@lnu.edu.ua}}
\vskip 0.5cm

\centerline{\footnotesize{March, 2019}}
\vskip 0.7cm

\footnotesize{\noindent\textbf{Abstract:} \textit{We constuct the theory of diagonalizability for matrices over Bezout rings of stable range 1 with the Kazimirsky condition. It is shown that a ring of stable range 1 with the right (left)  Kazimirsky condition is an elementary divisor ring if and only if it is a duo ring.} }

\vspace{1ex}
\footnotesize{\noindent\textbf{Key words and phrases:} \textit{ Bezout ring;  Kazimirsky condition;  elementary divisor ring;  stable range; duo ring.}

}

\vspace{1ex}
\noindent{\textbf{Mathematics Subject Classification}}: 06F20, 13F99.

\vspace{1,5truecm}

\normalsize


    Elementary divisor rings have been studied by many authors. The most full history of this rings can be found in \cite{1BVZmono2012}. Recall by I.~Kaplansky: a ring $R$ is called an elementary divisor ring if for an arbitrary matrix $A$ over $R$ there exist invertible matrices $P$ and $Q$ of suitable sizes   such that $PAQ=D$ is a diagonal matrix $D=(d_i)$ and $Rd_{i+1}R\subseteq Rd_i\cap d_iR$ for each $i$ \cite{5kaplansky1949}. In the case of commutative rings there are many developments on this rings, in the case of noncommutative rings they are  little investigated and fragmented. A general picture is far from its full description. Among these results one should especially note the result \cite{2zabkom1990} which shows that a distributive elementary divisor domain is duo. Tuganbaev extended this result in case of a distributive ring \cite{3tugan1991}.

    In \cite{4henriks1973} Henriksen proved that if $R$ is a unit regular ring than every matrix over $R$ admits diagonal reduction. Let's note that a regular ring is a unit regular if and only if this ring is a ring of stable range 1. Recall, that an associative ring $R$ with identity ($1\ne0$) is a ring of stable range 1 if for any $a,b\in R$ such that $aR+bR=R$ there exist $t\in R$ such that $(a+bt)R=R$\cite{5chen2011}. It is an open problem: when a ring of stable range 1 is an elementary divisor ring?

    In the case of commutative rings Kaplansky  announced that a commutative ring of stable range 1 is an elementary divisor ring if and only if it is a Bezout ring.

    An important role in studying  the elementary divisor rings are played  by the Hermite rings. A ring is called a right (left) Hermite ring if all $1\times 2$ ($2\times 1$) matrices have diagonal reduction, i.e. a ring $R$ is called a right (left) Hermite ring if for any $a,b\in R$ there exist invertible matrices $P$ and $Q$ such that $(a,b) P=(d,0)$ (\,$Q\binom{a}{b}=\binom{c}{0}$) for some elements $d,c\in R$.      An Hermite ring is a ring which is both right and left Hermite ring.

    The property of ring of being Hermite plays an important role in considering the possibility of diagonal reduction of matrices, since this condition is present in all known necessary and sufficient conditions of theorems on elementary divisor rings.

    Note that any Hermite ring is a finitely generated principal ideal ring, i.e. a ring in which an arbitrary finitely generated right or left ideal is principal (in modern terminology these rings are called Bezout rings). You should not think that an arbitrary Bezout ring of stable range 1 is an elementary divisor ring. Dubrovin \cite{6dubrov1986}   showed that any semilocal and semiprimary Bezout ring $R$ is an elementary divisor ring if and only if for any element $a\in R$ there exists $b\in R$ such that $RaR=bR=Rb$.  Note that a semilocal ring is a ring of stable range 1 and also note that a right (left) Bezout ring of stable range~1 is a right (left) Hermite ring \cite{5chen2011}.

    In modern research on the theory of rings a peculiar place is occupied by the rings in which every invertible element lies in the center of the ring, i.e. unit-central rings \cite{7khmsr2009}. These same authors raised an open question: is every unit-central ring with stable range~1 commutative?

    In this paper we consider Bezout rings of stable range 1 which Kazimirsky condition (which are generalize unit-central rings).
    \medskip

    Let $R$ be an associative ring with identity ($1\ne 0$).

\begin{definition}
    A ring $R$ is a ring with the right (left) Kazimirsky condition if for any $a\in R$ and any invertible element $u\in R$ the following inclusion holds $aR\supseteq uaR$ ($Ra\supseteq Rau$). If $R$ is a ring with the right and left Kazimirsky condition than we say, that $R$ is a ring with the Kazimirsky condition.
\end{definition}

    Obviously, an unit-central ring is a ring with the Kazimirsky condition.

\begin{proposition}\label{prop-1}
    Let $R$ be a ring of stable range 1. Then $Ra=Rau$ for any $a\in R$ and any invertible $u\in R$ if and only if $au=va$ for some invertible $v\in R$.
\end{proposition}

\begin{proof}
    If $au=va$ for some invertible $v\in R$, than obviously $Rau=Ra$.

    Let $Ra=Rau$. Since $R$ is a ring of stable range 1 than by \cite{1BVZmono2012} we have $va=au$ for some invertible $v\in R$.
\end{proof}

\begin{proposition}\label{prop-2}
    Let for any $a\in R$ and any invertible element  $u\in R$ there exists $x\in R$ such that $1+xa=u$. Then $Rau\subseteq Ra$ for any invertible $u\in R$ and any $a\in R$.
\end{proposition}

\begin{proof}
    Since $1+xa=u$, than $a+axa=au$ and $au\in Ra$, i.e. $Rau\subseteq Ra$.
\end{proof}

\begin{theorem}\label{theor-1}
    Let $R$ be a ring of stable range 1 with the left Kazimirsky condition. Than for any $a,b\in R$ if $aR+bR=R$, then $Ra+Rb=R$.
\end{theorem}

\begin{proof}
    Let $aR+bR=R$. Since $R$ is a ring of stable range 1 then $a+bt=u$ for some $t\in R$ and an invertible $u\in R$. Than $Ra+Rt=R$, i.e. $ax+t=w$ for some $x\in R$ and an invertible $w\in R$. Since $t=w-xa$, we have $u=a+b(w-xa)=a+bw-bxa=(a-bx)a+bw$. Since $R$ is a ring with the left Kazimirsky condition we have $Rbw\subseteq Rb$, i.e. $bw=yb$ for some $y\in R$. Than $u=(a-bx)a+yb$, i.e. $Ra+Rb=R$.
\end{proof}

    Recall that a ring $R$ is a right (left) quasi-duo ring if any right (left) maximal ideal is an ideals. A right and left quasi-duo ring is called a quasi-duo ring. By Theorem~\ref{theor-1} we have the following result.

\begin{proposition}\label{prop-3}
    A ring of stable range 1 with the left (right) Kazimirsky condition is a left (right) quasi-duo ring.
\end{proposition}

\begin{proof}
    Let $M$ be an arbitrary left maximal ideal. If $M$ is not two-sided then $mr\notin M$ for some $m\in M$ and $r\in R$. Therefore, $M+Rmr=R$ and than $n+xmr=1$ for some $n\in M$ and  $x\in R$. By Theorem~\ref{theor-1} we have $Rn+Rxm=R$.

    Since $m\in M$ than $xm\in M$, i.e. $Rn+Rxm\subseteq M$. This is of course impossible, so $M$ must be a two-sided ideal. By \cite{2zabkom1990} the Proposition is proved.
\end{proof}

    According to the results of \cite{3tugan1991,2zabkom1990} we have the following  result.

\begin{theorem}\label{theor-2}
    A Bezout ring of stable range 1 with the Kazimirsky condition is an elementary divisor ring if and only if it is a duo ring.
\end{theorem}

    Note that a duo ring is a ring with the Kazimirsky condition. Really, if $R$ is a duo ring and for any $a\in R$ and any invertible $u\in R$ we have $au=xa$ ($ua=ay$) for some $x,y\in R$, i.e. $Rau\subseteq Ra$ ($uaR\subseteq aR$).

    Recall that a ring $R$ is said to have the unit stable range 1 if $aR+bR=R$ implies that there exists  an invertible $u\in R$ such that $(a+bu)R=R$.

    We note that a regular ring $R$ has the unit stable range 1 if and only if it is unit-regular and for any idempotent $e,f\in R$, $eR+fR=R$ implies that there exist invertible $u,v\in R$ such that $eu+fv=1$. Moreover  $M_2(\mathbb{Z}_2)$ has the unit stable range 1, while $\mathbb{Z}_2$ does not have the unit stable range 1 \cite{5chen2011}.

    A ring of unit stable range 1 is a ring in which any additively generated by its units.

\begin{proposition}\label{prop-4}
    Let $R$ be a ring of unit stable range 1. Than for any nonzero $a\in R$ we have $a=u+w$ for some invertible $u,w\in R$.
\end{proposition}

\begin{proof}
    Let $a\in R\backslash \{0\}$. Than $aR+(-1)R=R$. Since $R$ is a ring of unit stable range 1 than exists an  invertible $u\in R$ such that $a-u=w$ is an invertible element of $R$, i.e. $a=u+w$.
 \end{proof}

\begin{proposition}\label{prop-5}
    A ring  of unit stable range 1 is a ring with the right (left) Kazimirsky condition if and only if it is a left (right) duo ring.
\end{proposition}

\begin{proof}
    Let $R$ be a ring of unit stable range 1, and a ring with right Kazimirsky condition and let $a,b\in R$. By Proposition~\ref{prop-4}, we have $a=u+w$ for some invertible $u,w\in R$. Since $R$ is a ring with the right Kazimirsky condition than $ubR\subseteq bR$ and $wbR\subseteq bR$, i.e. $ub=bx$ and $wb=by$ for some $x,y\in R$. Than $ab=(u+w)b=ub+wb=bx+by=b(x+y)$, i.e. $Rb\subseteq bR$, and $R$ is a left duo ring.
 \end{proof}

    From these results we have the following consequence.

\begin{theorem}\label{theor-3}
    A Bezout ring of unit stable range 1 with the Kazimirsky condition is an elementary divisor ring.
\end{theorem}

\noindent\textbf{Example.}
    In the article \cite{10marks2004} it was proved that if the skew polynomial ring $R[x; \sigma]$ is left or right duo, then $R[x; \sigma]$ is commutative. 

    Let $R = \mathbb{Z}[\sqrt{-7}]$ and $K = \mathbb{Q}(\sqrt{-7})$ be the quotient field of the ring $R$.
    Will consider an automorphism $\sigma\colon K\to K$ such that $\sigma(a + b\sqrt{-7}) = a - b\sqrt{-7}$, $a, b\in \mathbb{Q}$.  We will consider the subring $S$ of the ring $K[x;\sigma]$ which consists of polynomials the free member of which belongs $R$. The only units in the ring $S$ are 1 and $-1$, therefore consequently the ring $S$ is a unit-central ring but not commutative.

\end{document}